\documentclass[12pt]{amsart}

\usepackage{amsmath,amssymb,latexsym} 
\usepackage{fullpage}
\usepackage{etoolbox}
\usepackage{enumitem}
\usepackage{color}
\usepackage{frenchineq}
\usepackage{easylist}
\usepackage[colorlinks,linkcolor=blue,citecolor=magenta]{hyperref}
\usepackage[colorinlistoftodos,bordercolor=orange,backgroundcolor=orange!20,linecolor=orange,textsize=scriptsize]{todonotes}

\newtheorem{theorem}{Theorem}
\newtheorem*{theorem-cite}{Theorem}
\newtheorem{corollary}[theorem]{Corollary}

\newtheorem{lemma}[theorem]{Lemma}

\theoremstyle{definition}

\newtheorem{example}{Example}

\theoremstyle{remark}

\newtheorem*{claim}{Claim}
\newtheorem*{proof-claim}{Proof}

\newenvironment{changemargin}[2]{\begin{list}{}{%
\setlength{\topsep}{0pt}%
\setlength{\leftmargin}{0pt}%
\setlength{\rightmargin}{0pt}%
\setlength{\listparindent}{\parindent}%
\setlength{\itemindent}{\parindent}%
\setlength{\parsep}{0pt plus 1pt}%
\addtolength{\leftmargin}{#1}%
\addtolength{\rightmargin}{#2}%
}\item }{\end{list}}

\AtBeginEnvironment{proof-claim}{\vspace{-0.5cm}\footnotesize\begin{changemargin}{0.5cm}{0.5cm}}
\AtEndEnvironment{proof-claim}{\end{changemargin}}

\def\Ker{\operatorname{Ker}}
\def\H{\widetilde{H}}
\def\X{\mathsf{X}}
\def\L{\mathsf{L}}
\def\M{\mathsf{M}}
\def\N{\mathsf{N}}
\def\A{\mathsf{A}}
\def\B{\mathsf{B}}

\def\K{\mathsf{K}}
\def\R{\mathbb{R}}
\def\Z{\mathbb{Z}}
\def\F{\mathbb{F}}

\def\S{\mathcal{S}}
\def\C{\mathcal{C}}

\def\id{\operatorname{id}}

\def\id{\operatorname{id}}

\def\sd{\operatorname{sd}}

\def\sd{\operatorname{sd}}
\def\lk{\operatorname{lk}}
\DeclareMathOperator{\Ima}{Im}

\title{Different versions of the nerve theorem and rainbow simplices}  
\author{ Fr\'ed\'eric Meunier and Luis Montejano}
\address{F. Meunier, Universit\'e Paris Est, CERMICS, 77455 Marne-la-Vall\'ee CEDEX, France}
\email{frederic.meunier@enpc.fr}
\address{L. Montejano, Instituto de Matem\'aticas, National University of M\'exico at Quer\'etaro,
Juriquilla, Qro 76230}
\email{luismontej@gmail.com}


\begin{document}

\begin{abstract}
Given a simplicial complex and a collection of subcomplexes covering it, the {\em nerve theorem}, a fundamental tool in topological combinatorics, guarantees a certain connectivity of the simplicial complex when connectivity conditions on the intersection of the subcomplexes are satisfied.

We show that it is possible to extend this theorem by replacing some of these connectivity conditions on the intersection of the subcomplexes by connectivity conditions on their union.  While this is interesting for its own sake, we use this extension to generalize in various ways the Meshulam lemma, a powerful homological version of the Sperner lemma. We also prove a generalization of the Meshulam lemma that is somehow reminiscent of the polytopal generalization of the Sperner lemma by De Loera, Peterson, and Su. For this latter result,  we use a different approach and we do not know whether there is a way to get it via a nerve theorem of some kind.
\end{abstract}

\maketitle

\section{Introduction}\label{sec:intro}

The nerve theorem is a fundamental result in topological combinatorics. It has many applications, not only in combinatorics, but in category and homotopy theory and also  in applied and computational topology.  Roughly speaking, it relates the topological ``complexity'' of a simplicial complex to the topological ``complexity'' of the intersection complex of a ``nice'' cover of it. Stating the nerve theorem with conditions on intersections seems to be somehow dictated by its very nature. It might thus come as a surprise that a nerve theorem for unions also holds. 

In this paper, we prove such a theorem. Actually, we prove a theorem that interpolates between a version of the nerve theorem with intersections and a version with unions. Given a simplicial complex $\X$ and a finite collection of subcomplexes $\Gamma$, the {\em nerve} of $\Gamma$, denoted by $\N(\Gamma)$, is the simplicial complex with vertices the subcomplexes in $\Gamma$ and whose simplices are the subcollections of $\Gamma$ with a nonempty intersection.

We fix a field $\F$ throughout the paper.

\begin{theorem} \label{thmixed}
Consider a simplicial complex $\X$ with a finite collection $\Gamma$ of subcomplexes such that $\bigcup\Gamma=\X$.
Let $k$ and $\ell$ be two integers such that $-1\leq k\leq \ell<|\Gamma|$.
Suppose that the following two conditions are satisfied:
\begin{enumerate}[label=\textup{(\arabic*)}]
\item\label{inter} $\H_{k-|\sigma |}\big(\bigcap\sigma,\F\big)=0$ for every $\sigma\in\N(\Gamma)$ of dimension at most $k$.
\item\label{union} $\H_{|\sigma |-2} \big(\bigcup\sigma,\F\big)=0$ for every $\sigma\in\N(\Gamma)$ of dimension at least $k+1$ and at most $\ell$.
\end{enumerate}
Then $\dim\H_{\ell}(\N(\Gamma),\F)\leq\dim\H_{\ell}(\X,\F)$.
\end{theorem}

The case $k=\ell$ was obtained in \cite{MoO} and is a generalization of a classical version of the homological nerve theorem (see \cite[Theorem 6.1]{Mh2}). There are actually many ``classical'' versions, some of them with a homotopy condition in place of the homology condition. It seems that the oldest reference to a nerve theorem with a homology condition (actually an acyclicity condition) is due to Leray~\cite{Leray}. The case $k=-1$ is the nerve theorem for unions mentioned above. 

The second purpose of this paper is to provide a generalization of a related result -- Meshulam's lemma~\cite[Proposition 1.6]{Mh2} and~\cite[Theorem 1.5]{Mh1} -- which has several applications in combinatorics, such as the generalization of Edmonds' intersection theorem by Aharoni and Berger~\cite{AB}. Meshulam's lemma is a Sperner-lemma type result, dealing with coloured simplicial complexes and colourful simplices, and in which the classical boundary condition of the Sperner lemma is replaced by an acyclicity condition. Its original proof relies on a certain version of the nerve theorem and we show that Theorem~\ref{thmixed} can be used in the same vein to prove some variations of Meshulam's lemma.  We also prove -- with a completely different approach -- the following generalization of this lemma, which can be seen as a homological counterpart of the polytopal Sperner lemma by De Loera, Peterson, and Su~\cite{polytopal}, in a same way that Meshulam's lemma is a homological counterpart of the classical Sperner lemma. It can also be seen as a homological counterpart of Musin's Sperner-type results for pseudomanifolds~\cite{musin2015extensions} and of Theorem 4.7 in the paper by Asada et al.~\cite{asada2018fair}. We leave as an open question the existence of a proof based on a nerve theorem of some kind.

We recall that a {\em pseudomanifold} is a simplicial complex that is pure, non-branching (each ridge is contained in exactly two facets), and strongly connected (the dual is connected).  A {\em colourful simplex} in a simplicial complex whose vertices are partitioned into subsets $V_0,\ldots,V_m$ is a simplex with at most one vertex in each $V_i$.  Given a simplicial complex $\K$ and a subset $U$ of its vertices, $\K[U]$ is the subcomplex induced by $U$, i.e. the simplicial complex whose simplices are exactly the simplices of $\K$ whose vertices are all in $U$.

\begin{theorem}\label{thm:homol_sperner_gen}
Consider a simplicial complex $\K$ whose vertices are partitioned into $m+1$ subsets $V_0,\ldots,V_m$ and a nontrivial abelian group $A$. Let $\M$ be a $d$-dimensional pseudomanifold with vertex set $\{0,\dots,m\}$ such that $\H_d(\M,A)=A$.  Suppose that $\H_{|\sigma|-2}(\K[\bigcup_{i\in\sigma}V_i],A)=0$ for every $\sigma\in\M$. If $\H_d(\K,A)=0$ as well, then the number of $(d+1)$-dimensional colourful simplices in $\K$ is at least $m-d$.
\end{theorem}

The simplicial complexes are all abstract, the minimal dimension of a simplex is $0$ (i.e. the empty set is not a simplex), and for a simplicial complex $\X$ and any nontrivial abelian group $A$, we say that $\H_{-1}(\X,A)=0$ if and only if $\X$ is nonempty. When we use the word ``collection'', it means that repetition is allowed and the cardinality is counted with the repetitions.

 Theorems~\ref{thmixed} and~\ref{thm:homol_sperner_gen} are respectively proved in Sections~\ref{sec:mixed} and~\ref{sec:homol_sperner_gen}. Theorem~\ref{thmixed} is proved by induction. The base case, which is the case $k=-1$ (``nerve theorem for unions''), is proved in Section~\ref{sec:union}. Applications of Theorem~\ref{thmixed} are proposed in Section~\ref{sec:appli}. Some of these applications are new generalizations of Meshulam's lemma.

\section{A nerve theorem for unions}\label{sec:union}

In this section, we prove the following nerve theorem for unions.

\begin{theorem} \label{thmunion}
Consider a simplicial complex $\X$ with a finite collection $\Gamma$ of subcomplexes such that $\bigcup\Gamma=\X$.
Let $\ell$ be an integer such that $-1\leq\ell<|\Gamma|$.
Suppose that $\H_{|\sigma |-2} \big(\bigcup\sigma,\F\big)=0$ for every $\sigma\in\N(\Gamma)$ of dimension at most $\ell$. Then $\dim\H_{\ell}(\N(\Gamma),\F)\leq\dim\H_{\ell}(\X,\F)$.
\end{theorem}

This is the special case of Theorem~\ref{thmixed} when $k=-1$ and it will be used in Section~\ref{sec:mixed} to prove this latter theorem in its full generality.

If $\ell=-1$, the proof is easy: if $\X$ is nonempty, then $\N(\Gamma)$ is nonempty as well since $\Gamma$ covers $\X$. 
Let us thus consider the case where $\ell\geq 0$. The general structure of the proof, in particular the use of a carrier argument, shares similarities with the proof Bj\"orner proposed for his generalization of the nerve theorem~\cite{Bjo}.


We denote by $\N^{\ell}(\Gamma)$ the $\ell$-skeleton of $\N(\Gamma)$.

\begin{lemma}\label{lem:f}
There exists an augmentation-preserving chain map $f_{\sharp}\colon \C(\N^{\ell}(\Gamma),\F) \rightarrow \C(\sd\X,\F)$ such that for any $s$-dimensional simplex $\sigma\in\N(\Gamma)$ with $0\leq s\leq \ell$, the chain $f_{\sharp}(\sigma)$ is carried by $\sd\bigcup\sigma$.
\end{lemma}


\begin{proof}
Given a $0$-dimensional simplex $\{\A\}$ of $\N(\Gamma)$, the subcomplex is nonempty by definition (condition for $|\sigma|=1$) and there exists thus a vertex $v_{\A}$ in $\A$. We define $f_{\sharp}(\{\A\})$ to be $\{v_{\A}\}$. Note that $\{v_{\A}\}$ is a vertex of $\sd\A$. Suppose now that $f_{\sharp}(z)$ has been defined for every chain $z\in C_i(\N^{\ell}(\Gamma),\F)$ for $i$ up to $s-1<\ell$ and satisfies $\partial f_{\sharp}(z)=f_{\sharp}(\partial z)$ (where we use the augmentation map if $z$ is a $0$-chain). Suppose moreover that $f_{\sharp}(\sigma)\in C_i\left(\sd\bigcup\sigma,\F\right)$ for every $i$-dimensional simplex $\sigma\in\N(\Gamma)$ for $i$ up to $s-1$. Consider an $s$-dimensional simplex $\sigma$ of $\N(\Gamma)$. The chain $f_{\sharp}(\partial\sigma)$ has been defined, it belongs to $C_{s-1}\left(\sd\bigcup\sigma,\F\right)$ and $\partial f_{\sharp}(\partial\sigma)=0$. Since $s\leq \ell$, we have $\widetilde{H}_{\dim\sigma-1}\left(\bigcup\sigma,\F\right)=0$ and there exists a chain in $C_s\left(\sd\bigcup\sigma,\F\right)$ whose boundary is $f_{\sharp}(\partial\sigma)$. We define $f_{\sharp}(\sigma)$ to be this chain.
\end{proof}

For any simplex $\tau$ of $\X$, we set $\lambda(\tau)=\{\A\in\Gamma\colon \tau\in\A\}$.  

\begin{lemma}\label{lem:lambda}
The map $\lambda$ is a simplicial map $\sd\X\rightarrow\sd\N(\Gamma)$.
\end{lemma}

\begin{proof}
Let $\tau$ be a simplex of $\X$ and $\tau'$ any subset of $\tau$. We obviously have $\lambda(\tau)\subseteq\lambda(\tau')$. The map $\lambda$ reverses the order in the posets.
\end{proof}

We introduce now the chain map $\sd_{\sharp}\colon\C(\N(\Gamma),\F)\rightarrow\C(\sd\N(\Gamma),\F)$. Given $\sigma\in\N(\Gamma)$ of dimension $i$, the simplicial complex $\sd\sigma$ is a triangulation of $\sigma$ and $\sd_{\sharp}(\sigma)$ is the formal sum of all $i$-dimensional simplices of $\sd\sigma$, with the orientations induced by that of $\sigma$. It is well-known and easy to check that it is a chain map. We denote by $\sd_{\sharp}^{\ell}$ the restriction of $\sd_{\sharp}$ to $\C(\N^{\ell}(\Gamma),\F)$. Both $\sd_{\sharp}$ and $\sd_{\sharp}^{\ell}$ are augmentation-preserving.

We are going to show that there is a chain homotopy between $\lambda_{\sharp}\circ f_{\sharp}$ and $\sd_{\sharp}^{\ell}$. This will be done with the help of the {\em acyclic carrier theorem}, which we state here for sake of completeness.

An {\em acyclic carrier} from a simplicial complex $\K$ to a simplicial complex $\L$ is a function $\Psi$ that assigns to each simplex $\sigma$ in $\K$ a subcomplex $\Psi(\sigma)$ of $\L$ such that
\begin{itemize}
\item $\Psi(\sigma)$ is nonempty and acyclic
\item If $\tau$ is a face of $\sigma$, then $\Psi(\tau)\subseteq\Psi(\sigma)$.
\end{itemize}
A chain map $\mu\colon\C(\K,\F)\rightarrow\C(\L,\F)$ is {\em carried} by $\Psi$ if for each simplex $\sigma$, the chain $\mu(\sigma)$ is carried by the subcomplex $\Psi(\sigma)$ of $\L$.

\begin{theorem-cite}[Acyclic carrier theorem -- short version]
Let $\Psi$ be an acyclic carrier from $\K$ to $\L$. If $\phi$ and $\psi$ are two augmentation-preserving chain maps from $\C(\K,\F)$ to $\C(\L,\F)$ that are carried by $\Psi$, then $\phi$ and $\psi$ are chain-homotopic.
\end{theorem-cite}

The acyclic carrier theorem is usually stated for coefficients in $\Z$, and in a more general form~\cite[Theorem 13.3]{Mun}. The proof in this latter reference applies to arbitrary coefficients as observed by Segev~\cite[(1.2) p.667]{Seg93}.


In order to apply the acyclic carrier theorem, we define for each simplex $\sigma\in\N(\Gamma)$ the subcomplex $\Phi(\sigma)=\triangle\{\sigma'\in\N(\Gamma)\colon\sigma'\cap\sigma\neq\varnothing\}$. We denote by $\triangle(P)$ the order complex associated to a poset $P$.

\begin{lemma}\label{lem:carrier}
The map $\Phi$ is an acyclic carrier from $\N^{\ell}(\Gamma)$ to $\sd\N(\Gamma)$.
\end{lemma}

\begin{proof}
Let $\sigma$ be a simplex of $\N^{\ell}(\Gamma)$. The simplicial complex $\Phi(\sigma)$ is nonempty, and for any subset $\omega$ of $\sigma$, we have $\Phi(\omega)\subseteq\Phi(\sigma)$. The only thing that remains to be proved is thus the fact that $\Phi(\sigma)$ is acyclic. Actually, we have more: it is contractible. To see this, consider the following two simplicial maps $g,h:\Phi(\sigma)\rightarrow\Phi(\sigma)$ defined for any vertex $\sigma'\in V(\Phi(\sigma))$ by $g(\sigma')=\sigma'\cap\sigma$ and by $h(\sigma')=\sigma$ (constant map). Seeing $\Phi(\sigma)$ as the poset $\left(\{\sigma'\in\N(\Gamma)\colon\sigma'\cap\sigma\neq\varnothing\},\subseteq\right)$, we have $g\leq\id$ and $g\leq h$. By the order homotopy lemma~\cite[Lemma C.3]{Lon}, the maps $h$ and $\id$ are homotopic, which means that the identity map is homotopic to the constant map.
\end{proof}

\begin{proof}[Proof of Theorem~\ref{thmunion}] We first check that both $\lambda_{\sharp}\circ f_{\sharp}$ and $\sd_{\sharp}^{\ell}$ are carried by $\Phi$. Take $\sigma\in\N^{\ell}(\Gamma)$ of dimension $s$. Any simplex of $\sd\X$ in the support of $f_{\sharp}(\sigma)$ is of the form $\{\tau_0,\ldots,\tau_s\}$ with $\tau_0\subseteq\cdots\subseteq\tau_s$ and $\tau_i\in\bigcup_{\A\in\sigma}\A$ for all $i\in\{0,\ldots,s\}$ (see Lemma~\ref{lem:f}). In particular, $\lambda(\tau_i)\cap\sigma\neq\varnothing$ for all $i$ and $\{\lambda(\tau_0),\ldots,\lambda(\tau_s)\}$ is a simplex of $\triangle\{\sigma'\in\N(\Gamma)\colon\sigma'\cap\sigma\neq\varnothing\}$. It shows that $\lambda_{\sharp}\circ f_{\sharp}$ is carried by $\Phi$. Any simplex in the support of $\sd_{\sharp}^{\ell}(\sigma)$ is of the form $\{\sigma_0,\ldots,\sigma_s\}$ with $\sigma_0\subseteq\cdots\subseteq\sigma_s=\sigma$ and $\sigma_i\cap\sigma=\sigma_i\neq\varnothing$. Thus $\{\sigma_0,\ldots,\sigma_s\}$ is a simplex of $\triangle\{\sigma'\in\N(\Gamma)\colon\sigma'\cap\sigma\neq\varnothing\}$ and $\sd_{\sharp}^{\ell}$ is carried by $\Phi$.

Since both $\lambda_{\sharp}\circ f_{\sharp}$ and $\sd_{\sharp}^{\ell}$ preserve augmentation (here, we use the fact that $f_{\sharp}$ is augmentation-preserving -- see Lemma~\ref{lem:f}), we can apply the acyclic carrier theorem given above and there is a chain homotopy $D$ between $\lambda_{\sharp}\circ f_{\sharp}$  and $\sd_{\sharp}^{\ell}$. This chain homotopy is now used to conclude the proof.

For a simplicial complex $\K$, we denote by $Z_{\ell}(\K,\F)$ the cycle subspace of $C_{\ell}(\K,\F)$. Consider the map $j\colon  Z_{\ell}(\N(\Gamma),\F)\rightarrow Z_{\ell}(\N^{\ell}(\Gamma),\F)$ defined by $j(z)=z$ and the map
$$\begin{array}{rccc}\phi\colon & Z_{\ell}(\N(\Gamma),\F) & \longrightarrow & \H_{\ell}(\sd\X,\F) \\ & z & \longmapsto & [f_{\sharp}(j(z))],\end{array}$$ where the square brackets denote the homology class. The map $\phi$ is linear. The end of the proof consists in checking that $\Ker(\phi)$ is a subspace of $B_{\ell}(\N(\Gamma),\F)$, the boundary subspace of $C_{\ell}(\N(\Gamma),\F)$. From this, the conclusion follows immediately: 
\begin{align*}
\dim\H_{\ell}(\N(\Gamma),\F) & =  \dim Z_{\ell}(\N(\Gamma),\F)-\dim B_{\ell}(\N(\Gamma),\F) \\ & \leq  \dim Z_{\ell}(\N(\Gamma),\F)-\dim \Ker(\phi)
\\ & =  \dim \Ima(\phi) \\ 
& \leq  \dim\H_{\ell}(\sd\X,\F).
\end{align*}

Let $z$ be an element of $\Ker(\phi)$. It is such that $f_{\sharp}(j(z))=\partial c$ for some $c$ in $C_{\ell+1}(\sd\X,\F)$. By definition of $D$, we have $\partial D(j(z))+D(\partial j(z))=(\lambda_{\sharp}\circ f_{\sharp})(j(z))-\sd_{\sharp}^{\ell}(j(z))$, and hence $\partial D(j(z))=\partial \lambda_{\sharp}(c)-\sd_{\sharp}^{\ell}(j(z))$ since $j(z)=z$ is a cycle. Moreover, we have $\sd_{\sharp}^{\ell}(j(z))=\sd_{\sharp}(z)$. Therefore, $\sd_{\sharp}(z)=\partial(\lambda_{\sharp}(c)-D(j(z)))$. The algebraic subdivision theorem \cite[Theorem 17.2]{Mun} ensures that $\sd_*$ is an isomorphism for coefficients in $\Z$, and thus for coefficients in $\F$ (\cite[Theorem 51.1]{Mun}), which implies that $z$ is an element of $B_{\ell}(\N(\Gamma),\F)$.
\end{proof}

A version of Theorem~\ref{thmunion} where $\F$ is replaced by any finitely generated abelian group holds and can be obtained along the same lines. The conclusion in the theorem is no longer an inequality between the dimensions of the homology groups, which are then vector spaces, but between their ranks.

\section{The mixed nerve theorem}\label{sec:mixed}

The purpose of this section is to prove Theorem \ref{thmixed}. We will proceed by induction on $k$. The base case is given by Theorem~\ref{thmunion} which is the special case when $k=-1$ (only unions are considered), proved in Section~\ref{sec:union}. A crucial ingredient in the induction is a lemma that shows how to make homology of a simplicial complex vanish up to some dimension $d$, while keeping homology unchanged beyond $d$, by attaching simplices of dimension at most $d$. 

\begin{lemma}  \label{lem:killing}
Given a simplicial complex $\K$, there always exists a way to attach simplices of dimension at most $d$ to $\K$ to get a simplicial complex $\K'$ such that
\begin{itemize}
\item $\H_i(\K',\F)=0$ for all $i\leq d-1$.
\item $\H_i(\K',\F)\cong\H_i(\K,\F)$ for all $i\geq d$.
\end{itemize}
\end{lemma}

\begin{proof}
Let $n$ be the number of vertices of $\K$. We see $\K$ as a subcomplex of the standard $(n-1)$-dimensional simplex $\Delta_{n-1}$. Let $\K'$ be a simplicial complex obtained by adding to $\K$ the full $(d-1)$-skeleton $\Delta^{(d-1)}_{n-1}$ of $\Delta_{n-1}$ together with as many $d$-dimensional simplices of $\Delta_{n-1}$ as possible, so that $\H_i(\K',\F)\cong\H_i(\K,\F)$ for $i\geq d$. We claim that $\H_i(\K',\F)=0$ for $i\leq d-1$. This is clear for $i\leq d-2$ since $\K'\supseteq\Delta^{(d-1)}_{n-1}$. Suppose that $\H_{d-1}(\K',\F)\neq 0$. Then there exists a $d$-dimensional simplex $\sigma\in\Delta_{n-1}$ whose boundary $\partial\sigma$ is a nontrivial $(d-1)$-cycle of $\K'$. Let $\K''=\K'\cup\{\sigma\}$. Then any $d$-cycle of $\K''$ is also a $d$-cycle of $\K'$: if $z$ is a $d$-cycle of $\K''$, we can write $z$ as the sum $z'+\alpha \sigma$ with $z'\in C_d(\K',\F)$ and $\alpha\in\F$; then $\alpha\partial\sigma=-\partial z'$, which implies that $\alpha=0$ since $\partial\sigma$ is a nontrivial $(d-1)$-cycle of $\K'$ (here we use that $\F$ is a field). Thus, $\H_i(\K'',\F)\cong\H_i(\K',\F)\cong\H_i(\K,\F)$ for $i\geq d$ and $\H_i(\K'',\F)=0$ for $i\leq d-2$, contradicting the maximality of $\K'$. 
\end{proof}

Lemma~\ref{lem:killing} does not hold for integral homology. For example, if $\K$ is a triangulation of the real projective plane, then for any $2$-dimensional simplicial complex $\L$, either $\H_2(\K\cup\L,\Z)\neq 0=\H_2(\K,\Z)$, or $\H_1(\K\cup\L,\Z)\neq 0$. Indeed, if $\H_2(\K\cup\L,\Z)\cong\H_1(\K\cup\L,\Z)=0$, then the long exact sequence of the pair $(\K\cup\L,\K)$ implies that $\H_2(\K\cup\L,\Z)\cong\H_1(\K\cup\L,\Z)=\Z_2$, in contradiction with the fact that $\H_2(\K\cup\L,\K,\Z)$ is free (since $\K$ and $\L$ are $2$-dimensional).

For the proof of Theorem \ref{thmixed}, we also need the following technical  lemma.

\begin{lemma}  \label{lem:aux}
Consider a simplicial complex $\X$. Suppose given a nonempty finite collection $\sigma$ of subcomplexes such that $\H_{|\sigma|-|\tau|-1}(\bigcap\tau,\F)=0$ for every nonempty subcollection $\tau\subseteq\sigma$. Then $\H_{|\sigma|-2}\left(\bigcup\sigma,\F\right)=0$.
\end{lemma}

\begin{proof}
We start with a preliminary remark that will be used several times in the proof: we have $\bigcap\sigma\neq\varnothing$ (obtained with $\tau=\sigma$).

We prove actually that we have $\H_{|\sigma|-2}\left(\bigcup\sigma',\F\right)=0$ for every nonempty subcollection $\sigma'\subseteq\sigma$. We first prove this statement in the special case when $|\sigma'|=1$. In this case, letting $\tau=\sigma'$, the condition of the lemma imposes that $\H_{|\sigma|-2}(\bigcap\sigma',\F)=0$. Since $\sigma'$ has exactly one subcomplex, we have $\bigcap\sigma'=\bigcup\sigma'$, and the conclusion follows.

For the other cases, we proceed by induction on $m=|\sigma|+|\sigma'|$ and we start with the case $m=2$. In this case $\sigma'=\sigma$ and it is a collection of exactly one subcomplex, which is a case we have already treated.

Consider now the case $m\geq 3$. The case $|\sigma'|=1$ being known to be true, we assume that $|\sigma'|\geq 2$. We arbitrarily pick a subcomplex $\A$  in $\sigma'$. We introduce the collection
$$\omega=\big\{\A\cap \B\colon \B\in\sigma\setminus\{\A\}\big\}.$$ (Note that $\A\cap\B\neq\varnothing$ in this formula, since $\bigcap\sigma\neq\varnothing$.)
For every subcollection $\omega_0\subseteq\omega$, there is a $\sigma''\subseteq\sigma\setminus\{\A\}$ such that $\omega_0=\{\A\cap\B\colon
\B\in\sigma''\}$. 
This means that for every nonempty subcollection $\omega_0\subseteq\omega$, we have $\H_{|\sigma|-(|\sigma''|+1)-1}(\bigcap\omega_0,\F)=0$ by the condition of the lemma applied with $\tau=\sigma''\cup\{\A\}$, and thus $\H_{|\omega|-|\omega_0|-1}(\bigcap\omega_0,\F)=0$. The collection $\omega$ satisfies thus the condition of the lemma. The collection $\omega'=\{\A\cap\B\colon\B\in\sigma'\setminus\{\A\}\}$ is a nonempty subcollection of $\omega$ and we have $|\omega|+|\omega'|=m-2$. Hence the induction applies and we get $\H_{|\omega|-2}\left(\bigcup\omega',\F\right)=0$, which means
$$
\H_{|\sigma|-3}\left(\A\cap\left(\bigcup_{\B\in\sigma'\setminus\{\A\}}\B\right),\F\right)=0.
$$ By the Mayer-Vietoris exact sequence of the pair $\left(\bigcup_{\B\in\sigma'\setminus\{\A\}}\B,\A\right)$, we have that $\H_{|\sigma|-2}\left(\bigcup\sigma',\F\right)=0$, as required, because the condition of the lemma with $\tau=\{\A\}$ imposes that $\H_{|\sigma|-2}(\A,\F)=0$ and induction shows that $\H_{|\sigma|-2}\left(\bigcup_{\B\in\sigma'\setminus\{\A\}}\B,\F\right)=0$. (Since $\bigcap\sigma\neq\varnothing$, we have $\A\cap\left(\bigcup_{\B\in\sigma'\setminus\{\A\}}\B\right)\neq\varnothing$ and the Mayer-Vietoris exact sequence holds in the reduced case.)
\end{proof}

\begin{proof}[Proof of Theorem~\ref{thmixed}]
The proof is by induction on $k$. The integer $\ell$ is considered as fixed. If $k=-1$, then it follows directly from Theorem~\ref{thmunion}. Suppose our theorem is true for $k-1$. We shall prove it for $k$.

Consider a simplicial complex $\X$ with a finite collection $\Gamma$ of subcomplexes as in the statement of the theorem we want to prove.

\begin{claim} For every nonnegative integer $j\leq k-1$, there exists a simplicial complex $\X(j)$, obtained by attaching to $\X$ simplices of dimension at most $k-j-1$, and a finite collection $\Gamma(j)$ of subcomplexes such that $\bigcup\Gamma(j)=\X(j)$ and such that
\begin{enumerate}[label=\emph{\alph*})]
\item\label{a} $\N(\Gamma)=\N(\Gamma(j))$.
\item\label{b} $\H_{k-|\tau|}\big(\bigcap\tau,\F\big)=0$ for every $\tau\in\N(\Gamma(j))$ of dimension at most $k-1$.
\item\label{c} $\H_{|\tau|-2} \big(\bigcup\tau,\F\big)=0$ for every $\tau\in\N(\Gamma(j))$ of dimension at least $k+1$ and at most $\ell$.
\item\label{d} $\H_{k-|\tau|-1}\big(\bigcap\tau,\F\big)=0$ for every $\tau\in\N(\Gamma(j))$ of dimension at least $j$ and at most $k-1$.
\end{enumerate}
\end{claim}

We prove the claim by decreasing induction on $j$, starting with the base case $j=k-1$. This case is obviously true since in that case, we do not even have to add any simplex, and we set $\Gamma(k-1)=\Gamma$.

Consider the case $j\leq k-2$. We start with $\X(j+1)$ and $\Gamma(j+1)$, which we know to exist. Consider a $\sigma$ in  $\N(\Gamma(j+1))$ of dimension exactly $j$ such that $\H_{k-j-2}\big(\bigcap\sigma,\F\big)\neq 0$. If such a simplex does not exist, we are done: we set $\Gamma(j)=\Gamma(j+1)$. So, suppose that such a $\sigma$ exists. By Lemma~\ref{lem:killing}, we can attach a collection $\C$ of simplices of dimension at most $k-j-1$ to $\bigcap\sigma$ so that $$\H_{k-j-2}\big(\bigcap\sigma',\F\big)=0\quad\mbox{and}\quad\H_{k-j-1}\big(\bigcap\sigma',\F\big)=0,$$ where $\sigma'=\{\A\cup\C\colon\A\in\sigma\}$. We get the right-hand side equality as a consequence of $\H_{k-j-1}\big(\bigcap\sigma',\F\big)=\H_{k-j-1}\big(\bigcap\sigma,\F\big)$.
Define $$\X'=\X(j+1)\cup\C\quad\mbox{and}\quad\Gamma'=(\Gamma(j+1)\setminus\sigma)\cup\sigma'.$$ 
Property~\ref{a} is automatically satisfied for $\Gamma'$: for $\tau\in\Gamma(j+1)$, if $\tau\setminus\sigma\neq\varnothing$, then the corresponding simplex $\tau'$ in $\Gamma'$ is such that $\bigcap\tau=\bigcap\tau'$; if $\tau\subseteq\sigma$, then both $\bigcap\tau$ and $\bigcap\tau'$ are nonempty.

Consider a simplex $\tau'\in\N(\Gamma')$ of dimension at most $k-1$. Denote by $\tau$ the corresponding simplex in $\Gamma(j+1)$.  If $|\tau'|\leq j$, property \ref{b} is satisfied since we have added simplices of dimension at most $k-j-1$. If $|\tau'|\geq j+1$, either $\tau=\sigma$, in which case $\H_{k-|\tau|}\big(\bigcap\tau',\F\big)=0$, or $\tau\neq\sigma$, in which case as above $\bigcap\tau=\bigcap\tau'$. In both cases, property \ref{b} is satisfied. 

Property \ref{c} is satisfied because of the dimension of the attached simplices. 

We repeat this operation as many times as necessary to satisfy property \ref{d}. Note that when we attach $\C$, we do not alter the satisfaction of property \ref{d} for the simplices $\tau$ that already satisfy it: as above, when $\tau\neq\sigma$, we have $\bigcap\tau=\bigcap\tau'$.

At the end of the process, we get a simplicial complex $\X(j)$ and a finite collection $\Gamma(j)$ of subcomplexes covering it, so that properties \ref{a}, \ref{b}, \ref{c}, and~\ref{d} are simultaneously satisfied.

The claim is proved.

\medskip

We can now finish the proof of Theorem~\ref{thmixed}. Apply the claim for $j=0$. It ensures the existence of a simplicial complex $\X(0)$ and a collection $\Gamma(0)$ of subcomplexes covering it, with the properties \ref{a}--\ref{d} satisfied. We want to apply Theorem~\ref{thmixed} for $k-1$. Consider a $\sigma\in\N(\Gamma(0))$ of dimension at most $k-1$. We have $\H_{k-|\sigma|-1}(\bigcap\sigma,\F)=0$ because of property~\ref{d}. Hence, condition~\ref{inter} is satisfied. 
Consider now a $\sigma\in\N(\Gamma(0))$ of dimension exactly $k$. Every strict subset $\tau$ of $\sigma$ is such that $\H_{k-|\tau|}(\bigcap\tau,\F)=0$ because of \ref{b}, and $\H_{-1}(\bigcap\sigma,\F)=0$ because $\bigcap\sigma\neq\varnothing$. Lemma~\ref{lem:aux} implies then that $\H_{k-1}(\bigcup\sigma,\F)=0$. Together with \ref{c}, it implies that condition~\ref{union} is satisfied. 

The simplicial complex $\X(0)$ and the collection $\Gamma(0)$ satisfy the induction hypothesis for $k-1$ and therefore $\H_{\ell}(\N(\Gamma(0)),\F)=\H_{\ell}(\N(\Gamma),\F)$ is of dimension at most $\dim\H_{\ell}(\X(0),\F)$. Since $\X(0)$ has been obtained from $\X$ by attaching simplices of dimension at most $k-1<\ell$, we have $\H_{\ell}(\X(0),\F)=\H_{\ell}(\X,\F)$, and the conclusion follows.
\end{proof}

\section{Applications of the mixed nerve theorem}\label{sec:appli}
 
\subsection{Homological Helly-type results} 

We start our applications with the following Helly-type result.

\begin{theorem}  \label{thmH}
Consider a simplicial complex $\X$ with a finite collection $\Gamma$ of subcomplexes.
Let $k$ be an integer such that $-1\leq k\leq|\Gamma|-2$.
Suppose that the following two conditions are satisfied for every subcollection $\Gamma'\subseteq\Gamma$:
\begin{enumerate}[label=\textup{(\arabic*h)}]
\item\label{inter-h} $\H_{k-|\Gamma'|}\big(\bigcap \Gamma',\F\big)=0$ whenever   $1\leq |\Gamma'| \leq k+1$.
\item\label{union-h} $\H_{|\Gamma'|-2} \big(\bigcup \Gamma',\F\big)=0$ whenever $k+2\leq |\Gamma'|\leq|\Gamma|.$
\end{enumerate}
Then $\bigcap\Gamma\neq\varnothing$.
\end{theorem}

\begin{proof}  Suppose for a contradiction that there exists a nonempty $\Gamma'\subseteq\Gamma$ such that $\bigcap\Gamma'=\varnothing$. Choose such a $\Gamma'$ of minimal cardinality. Because of condition~\ref{inter-h}, we have $|\Gamma'|\geq k+2$.  Since $\H_{|\Gamma'|-2}(\bigcup\Gamma',\F)=0$ by condition~\ref{union-h}, Theorem~\ref{thmixed} for $\X'=\bigcup\Gamma'$ and $\ell=|\Gamma'|-2$ implies that $\H_{\ell}(\N(\Gamma'),\F)=0$. By definition of $\Gamma'$, the $\ell$-skeleton of $\N(\Gamma')$ is the boundary of the $(\ell+1)$-dimensional simplex and is thus homeomorphic to $\S^{\ell}$. The fact that the $\ell$-th homology group of $\N(\Gamma')$ vanishes implies that there is at least one $(\ell+1)$-dimensional simplex in $\N(\Gamma')$, i.e. that we have $\bigcap\Gamma'\neq\varnothing$, which is a contradiction.
\end{proof}

\begin{example}\label{ex-thmH}
Figure~\ref{fig:thmH} illustrates Theorem~\ref{thmH} with $\X$ being a triangulation of a $3$-dimensional cylinder and with $\Gamma$ containing four subcomplexes denoted by A, B, C, and D. The figure shows the situation in the front disk (with X$=$C) and in the back disk (with X$=$D). The subcomplexes C and D have some thickness but otherwise the cylinder is filled with A and B, the subcomplex A occupying the top-half and the subcomplex B the bottom-half (except for the parts occupied by C and D). Consider for instance the case $\F=\F_2$. The $0$-th homology group of any pair of subcomplexes is zero, except for the pair $\{\text{C,D}\}$. The $1$-th homology group is zero for any triple of subcomplexes. The $2$-th homology group is zero for the full collection $\Gamma$, which has moreover an empty intersection.Theorem~\ref{thmH} for $k=-1$ implies that if we add any curve inside B (avoiding A) to connect C to D (so as to make the $0$-th homology group of that pair vanish), we create a non-contractible cycle in $\text{A}\cup\text{C}\cup\text{D}$.
\end{example}

\begin{figure}[h]
\label{fig:thmH}
\begin{center}
\includegraphics[width=6cm]{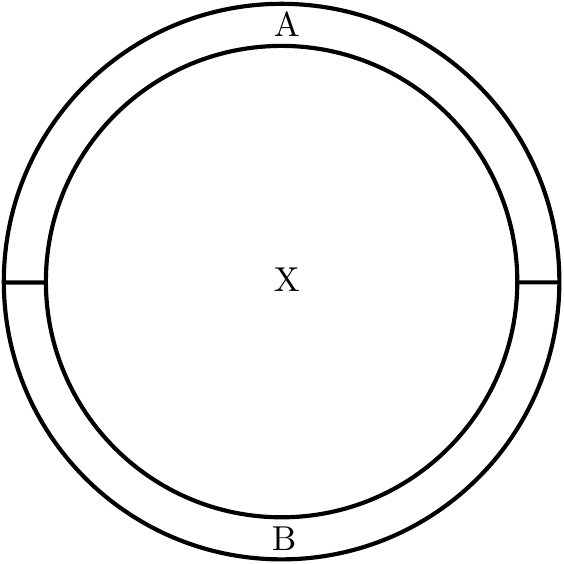}
\caption{The figure used in Example~\ref{ex-thmH} to illustrate Theorem~\ref{thmH}.}
\end{center}
\end{figure}

Note that when $\X$ is embedded in $\R^d$, Theorem~\ref{thmH} gives rise to the following two corollaries. The second one is the already known ``topological Helly theorem''~\cite{MoO} (see also \cite{KM}). The first one is new and could be seen as a ``topological Helly theorem'' for unions.

\begin{corollary} Consider a simplicial complex $\X$ embedded in $\R^d$ with a finite collection $\Gamma$ of subcomplexes.
Suppose that we have $\H_{|\Gamma'|-2}( \bigcup\Gamma',\F)=0$ for every nonempty subcollection $\Gamma'\subseteq\Gamma$ of cardinality at most $d+1$. Then 
$\bigcap \Gamma\neq \varnothing$.
\end{corollary}

\begin{corollary} Consider a simplicial complex $\X$ embedded in $\R^d$ with a finite collection $\Gamma$ of subcomplexes such that $|\Gamma|\geq d+2$.
Suppose that we have$\H_{d-|\Gamma'|}( \bigcap\Gamma',\F)=0$ for every subcollection $\Gamma'\subset\Gamma$ of cardinality at most $d+1$. Then 
$\bigcap \Gamma\neq \varnothing$.
\end{corollary}

\subsection{Homological Sperner-type results}

All results of this subsection deal with a simplicial complex $\K$ whose vertices are partitioned into subsets $V_0,\ldots,V_m$  as colours and ensure the existence of a rainbow simplex under some homological condition.  Formally, a {\em rainbow simplex} in such a simplicial complex is a simplex that has exactly one vertex in each $V_i$. We show how the nerve theorems introduced in the present paper can be used to get results of this type. We were however not able to prove Theorem~\ref{thm:homol_sperner_gen} within this framework (even with the field $\F$ in place of the abelian group $A$) and a proof with a completely different approach is given in Section~\ref{sec:homol_sperner_gen}.

For $S\subseteq\{0,\dots, m\}$, we denote by $\K_S$ the subcomplex of $\K$ induced by the vertices in $\bigcup_{i\in S}V_i$. The following theorem is a generalization of the Meshulam lemma cited in the introduction.

\begin{theorem}\label{thm:homol_sperner}
Consider a simplicial complex $\K$ whose vertices are partitioned into $m+1$ subsets $V_0,\ldots,V_m$. Suppose that $\H_{|S|-2}(\K_S,\F)=0$ for every nonempty $S\subseteq\{0,\dots,m\}$. Then there exists at least one rainbow simplex in $\K$.
\end{theorem}

Meshulam's lemma is the same statement with the stronger requirement that $\K_S$ is \break $({|S|-2})$-acyclic instead of $\widetilde{H}_{|S|-2}(\K_S,\F)=0$. Theorem~\ref{thm:homol_sperner} has been recently introduced by the second author~\cite{Mo1}.  We present here a new proof showing that it is a consequence of our ``union'' version of the nerve theorem (Theorem~\ref{thmunion}). Right after this proof, we will present a generalization of Theorem~\ref{thm:homol_sperner}, with yet another proof.

We need a preliminary lemma. For an integer $i\in\{0,\dots,m\}$, we define $\A_i$ to be subcomplex of $\sd\K$ induced by the vertices $\tau\in V(\sd\K)$ such that $\tau\cap V_i\neq\varnothing$.

\begin{lemma}\label{lem:S}
For any nonempty $S\subseteq \{0,\dots,m\}$, the simplicial complexes $\bigcup_{i\in S}\A_i$ and $\K_S$ have same homology groups.
\end{lemma}

\begin{proof}
For $\tau\in V\left(\bigcup_{i\in S}\A_i\right)$, we define $\lambda(\tau)$ to be $\tau\setminus\bigcup_{i\notin S}V_i$. It induces a simplicial map $\bigcup_{i\in S}\A_i\rightarrow\sd\K_S$. Now, consider the simplicial inclusion map $j:\sd\K_S\rightarrow\bigcup_{i\in S}\A_i$. We have clearly $\lambda\circ j=\id_{\sd\K_S}$. We also have $j\circ\lambda\leq\id_{\bigcup_{i\in S}\A_i}$ (with the order-preserving map point of view; see the proof of Lemma~\ref{lem:carrier}). The order homotopy lemma~\cite[Lemma C.3]{Lon} implies then that $j$ and $\lambda$ are homotopy inverse.
\end{proof}

\begin{proof}[Proof of Theorem~\ref{thm:homol_sperner}]
We prove by induction on $|S|$ that $\bigcap_{i\in S}\A_i\neq\varnothing$ for any nonempty $S\subseteq\{0,\ldots,m\}$. Since $\widetilde{H}_{-1}(\K_S,\F)=0$ for any singleton $S\subseteq \{0,\dots,m\}$, every $\A_i$ is nonempty. It proves that the above statement is correct for $|S|=1$. Consider now a set $S\subseteq \{0,\dots,m\}$ of cardinality $s\geq 2$. We denote by $\N_S$ the nerve of $\{\A_i\colon i\in S\}$. By induction, $\N_S$ contains the boundary of the $(s-1)$-dimensional simplex (with vertex set $\{\A_i\colon i\in S\}$).
According to Lemma~\ref{lem:S}, we have $\widetilde{H}_{|T|-2}(\bigcup_{i\in T}\A_i,\F)=0$ for every nonempty subset $T$ of $S$. Theorem~\ref{thmunion} with $\X=\bigcup_{i\in S}\A_i$,  $\Gamma=\{\A_i\colon i\in S\}$, and $\ell=s-2$ implies then that $\widetilde{H}_{s-2}(\N_S,\F)=0$ and in particular that there is at least one $(s-1)$-dimensional simplex in $\N_S$: this latter simplicial complex is thus exactly the $(s-1)$-dimensional simplex and $\bigcap_{i\in S}\A_i\neq\varnothing$.

To conclude, note that any vertex of $\bigcap_{i=0}^m\A_i$ is a simplex of $\K$ intersecting every $V_i$.
\end{proof}

Theorem~ \ref{thm:homol_sperner} tells us that the responsibility for the existence of a  rainbow simplex is due to the homology of the subcomplexes $\K_S$. Next we shall see that this responsibility can be shared by other subcomplexes. For $S\subseteq \{0,\ldots,m\}$, we denote by $\tilde \K_S$ the subcomplex of 
$\K$ consisting of  those simplices $\sigma$ of $\K$ for which the subset of colours assigned to $\sigma$ does not contain $S$.  

\begin{theorem}\label{remixed}
 Consider a simplicial complex $\K$ whose vertices are partitioned into $m+1$ subsets $V_0,\ldots,V_m$ and such that $\widetilde H_{m-1}(\K,\F)=0$. Let $k$ be an integer such that $-1\leq k\leq m-1$.  Suppose that for every $S\subseteq\{0,\ldots,m\}$:
 \begin{enumerate}[label=\textup{(\arabic*s)}]
\item\label{inter-s} $\H_{k-|S|}( \K_{\{0,\ldots,m\}\setminus S},\F)=0$  whenever $1\leq |S| \leq k+1$.
\item\label{union-s} $\H_{|S|-2} (\tilde \K_S,\F)=0$ whenever $k+2 \leq |S| \leq m$.
\end{enumerate} 
 Then $\K$ contains a rainbow simplex.
\end{theorem}

\begin{proof}
Let $I=\{0,\ldots,m\}$.
For $i\in I$, we define $\B_i$ to be $\K_{ I\setminus\{i\} }$. Note that 
$$\bigcap_{i\in S}\B_i=\K_{I\setminus S}\qquad\mbox{and}\qquad\bigcup_{i\in S}\B_i=\tilde \K_S.$$
Note also that $\K_I=\K$, $\tilde \K_{\{i\}}=\B_i$ and the inclusionwise minimal simplices of $\K$ not in $\tilde \K_I$ are the rainbow simplices.

Using these remarks, it is easy to check that conditions~\ref{inter-s} and~\ref{union-s} imply that $\X=\K$ and $\Gamma=\{\B_i\colon i\in I\}$ satisfy conditions~\ref{inter-h} and~\ref{union-h} of Theorem~\ref{thmH}, except for $|\Gamma'|=|\Gamma|=m+1$. Since $\bigcap_{i\in I}\B_i =\varnothing$, we have thus $\H_{m-1} ( \bigcup_{i\in I}\B_i,\F )\not= 0$. Note that  $\bigcup_{i\in I}\B_i = \tilde \K_I$, hence $ \tilde \K_I$ is different from $\K$ because by hypothesis $\widetilde H_{m-1}(\K,\F)=0$. Consequently there must be a rainbow simplex in $\K$. This completes the proof of  our theorem.
\end{proof}

Theorem \ref{remixed} for $k=m-1$ is exactly Theorem~\ref{thm:homol_sperner}. When $k=-1$, we get the following corollary.

\begin{corollary}\label{cor:remixed}
Consider a simplicial complex $\K$ whose vertices are partitioned into $m+1$ subsets $V_0,\ldots,V_m$ and such that $\H_{m-1}(\K,\F)=0$. Suppose that for every $S\subseteq\{0,\ldots,m\}$, we have $\H_{|S|-2} (\tilde \K_S,\F)=0$ whenever $1\leq |S| \leq m$. Then $\K$ contains a rainbow simplex.
\end{corollary}

\begin{example}
A {\em totally dominating set} in a graph $H$ is a subset $D$ of its vertices such that any vertex has a neighbour in $D$ (if there is no loop, a vertex is not a neighbour of itself). The minimum cardinality of a totally dominating set is the {\em total domination number} of $H$ and is denoted $\widetilde\gamma(H)$. 

We have:

\medskip

{\em Consider a graph $G$ with a three-colouring of the vertices  such that, for any fixed pair of colours, removing all edges with one endpoint of each colour does not disconnect the graph. The colouring does not need to be proper but there must be at least one vertex from each color. If $\widetilde\gamma(\overline{G})\geq 5$, then there is triangle in $G$ with the three colours.}

\medskip

This is a direct consequence of Corollary~\ref{cor:remixed} with $\K$ being the clique complex of $G$ and $m=2$. (The {\em clique complex} of a graph is the simplicial complex whose vertices are the vertices of the graph and whose simplices are its cliques.) The fact that $\widetilde\gamma(\overline{G})$ is at least $5$ implies that the clique complex of $G$ is $1$-connected (see, e.g.,~\cite[Section 2]{ABZ07}).
\end{example}

Corollary~\ref{cor:remixed} (and thus Theorem~\ref{remixed}) is not true if we remove the condition $\H_{m-1}(\K,\F)=0$. A counter-example for $m=2$ is obtained with $\K$ being the triangulation of a torus and with each $V_i$ inducing a non-contractile strip.

We end this subsection with further consequences of Theorem \ref{thm:homol_sperner}. We present them here since they are related results that we obtained while working on our two main results, Theorems~\ref{thmixed} and~\ref{thm:homol_sperner_gen}, but we do not need any of these latter to establish them. Complementary results with a similar flavour have been recently obtained by the second author~\cite{Mo2}.

Let $\K$ be a simplicial complex whose vertex set is partitioned into colour sets $V_0,\ldots,V_m$. We say that  a vertex $v\in V_i$ is \emph{isolated on its colour} if $v$ is an isolated point in $\K_{\{i\}}$.

\begin{theorem}\label{thm:isolated}
Let $\K$ be a simplicial complex whose vertex set is partitioned into colour sets $V_0,\ldots,V_m$ and suppose the vertex $v\in V(\K)$ is isolated on its colour. If $\widetilde{H}_{|S|-2}(\K_S,\F)=0$ for every nonempty $S\subseteq \{0,\dots,m\}$, then there exists a rainbow simplex $\sigma$ in $\K$ containing the vertex $v$.
\end{theorem}

\begin{proof} Consider the following subcomplex of $\K$, called the {\em link} of $v$:
$$\lk(v,\K)=\{\sigma\in \K\colon v \notin \sigma \mbox { and } ( \{v\}\cup\sigma) \in\K\}.$$ 
Suppose without loss of generality $v \in \K_{\{0\}}$. Since $v$ is isolated on its color, then the vertices of $\lk(v,\K)$ are partitioned into $m$ color classes $V^\prime_1,\dots,V^\prime_m$ with $V_i^\prime\subseteq V_i$.  For $S\subseteq \{1,\dots, m\}$, we denote by $\lk(v,\K)_S$ the subcomplex of $\lk(v,\K)$ induced by $\bigcup_{i\in S}V^\prime_i$. Note that we have $\lk(v,\K)_S=\lk(v,\K) \cap \K_S$.
We wish to prove that for every nonempty subset $S\subseteq\{1,\dots,m\}$ we have $\H_{|S|-2}(\lk(v,\K)_S,\F)=0$,  which will imply, by Theorem \ref{thm:homol_sperner}, that there is a rainbow simplex that contains $v$.

Let us consider the Mayer-Vietoris exact sequence of the pair $\big(\K_S, v*\lk(v,\K)_S)$: 
$$\cdots\to\H_{|S|-1}(\K_S,\F)\oplus \H_{|S|-1}(v*\lk(v,\K)_S,\F)\to \H_{|S|-1}(\K^\prime,\F) \to \H_{|S|-2}(\lk(v,\K)_S,\F) \to 0\to\cdots$$
where $\K_S\cap\,(v*\lk(v,\K)_S)=\lk(v,\K)_S$ and $\K^\prime=\K_S\cup\,( v*\lk(v,\K)_S).$
Consequently, we have $\H_{|S|-2}(\lk(v,\K)_S,\F)=0$, provided the homomorphism $\H_{|S|-1}(\K_S,\F)\to \H_{|S|-1}(\K^\prime,\F)$
induced by the inclusion is an epimorphism.  

For that purpose let us consider $\K^{\prime\prime}$ be the subcomplex of $\K$ induced by the vertices in $\bigcup_{i\in S}V_i \cup (V_0\setminus\{v\})$.  Note that $\K^{\prime}$ is the subcomplex of $\K$ induced by the vertices in $\bigcup_{i\in S}V_i \cup \{v\}$. Since $v$ is isolated in $\K_{\{0\}}$, we have that $\K^{\prime\prime}\cup \K^\prime=\K_{S\cup\{0\}}$ and $\K^{\prime\prime}\cap \K^\prime=\K_S.$ The Mayer-Vietoris exact sequence of the pair $\big(\K^{\prime\prime}, \K^\prime\big)$ is
$$\cdots\to\H_{|S|-1}(\K_S,\F)\to \H_{|S|-1}(\K^\prime,\F)\oplus \H_{|S|-1}(\K^{\prime\prime},\F)  \to \H_{|S|-1}(\K_{S\cup\{0\}},\F)=0\to\cdots$$
which implies that the homomorphism $\H_{|S|-1}(\K_S,\F)\to \H_{|S|-1}(\K^\prime,\F)$
induced by the inclusion is an epimorphism as we wished. 
\end{proof}

\begin{corollary}\label{discrete}
Let $\K$ be a simplicial complex whose vertex set is partitioned into colour sets $V_0,\ldots,V_m$ and suppose $\K_{\{i\}}$ is  $0$-dimensional for $i\in \{0,\ldots, m\}$.  If $\H_{|S|-2}(\K_S,\F)=0$ for every nonempty $S\subseteq \{0,\ldots,m\}$, then every simplex is contained in a rainbow simplex.
\end{corollary}

\begin{proof}
The proof is by induction on $n$, the dimension of the simplex $\sigma$ for which we want to prove containment in a rainbow simplex. If $n=0$, then the corollary follows from Theorem \ref{thm:isolated}. Suppose the corollary is true for $n-1$, we shall prove it for $n$. Suppose 
$\sigma=\{v_0,\ldots,v_n\}$. Then, by the proof of Theorem~\ref{thm:isolated}, $\lk(v_0,\K)$ satisfies the hypothesis of the corollary. By induction $\{v_1,\ldots,v_n\}$ is contained is a rainbow simplex 
$\{v_1,\ldots,v_n,\ldots, v_m\}$ of $\lk(v_0,\K)$. Consequently $\{v_0,\dots,v_m\}$ is a rainbow simplex of $\K$ containing $\sigma$.
\end{proof}

\section{A polytopal generalization of Meshulam's lemma}\label{sec:homol_sperner_gen}

This section is devoted to a proof of Theorem~\ref{thm:homol_sperner_gen}. Note that Theorem~\ref{thm:homol_sperner} is the special case where $d=m$ and $\M$ is the boundary of the $m$-dimensional simplex with vertex set $\{0,\dots,m\}$.

The following counting lemma will be used in the proof of that theorem. The {\em supporting complex} of a chain is the simplicial complex whose simplices are all simplices in the support of the chain as well as their faces.

\begin{lemma}\label{lem:count}
Let $A$ be a nontrivial abelian group. Consider a simplicial complex $\L$ and a chain $c\in C_s(\L,A)$ such that the supporting complex of $\partial c$ is a pseudomanifold with $n$ vertices. Then the support of $c$ is of cardinality at least $n-s$.
\end{lemma}

\begin{proof}
Consider the graph $G(c)$ whose vertices are the $s$-dimensional simplices in the support of $c$ and whose edges connect two simplices having a common facet. For a connected component $K$ of $G(c)$, we denote by $c_K$ the chain obtained from $c$ by keeping only the $s$-dimensional simplices corresponding to vertices in $K$. Since the supporting complex of $\partial c$ is a pseudomanifold, it is strongly connected (see the definition in Section~\ref{sec:intro}) and only one connected component $K_0$ is such that $\partial c_{K_0}$ is nonzero. The supporting complex of $c_{K_0}$ has at least $n$ vertices.

We prove now that any chain $c'$ in $C_s(\L,A)$, such that $G(c')$ is connected and whose supporting complex has at least $n$ vertices, has a support of cardinality at least $n-s$. This implies then directly the desired result. The proof works by induction on the cardinality $k$ of the support of $c'$. If $k=1$, the statement is obviously true: $s+1-s=1$. Suppose that $k>1$. In a connected graph with at least one edge, there is at least one vertex whose removal does not disconnect the graph. We can thus remove a simplex from the support of $c'$ and obtain a new chain $c''$ such that $G(c'')$ is still connected. Note that the removed simplex has a facet in common with a simplex in the support of $c''$. It means that at most one vertex has been removed from the supporting complex of $c'$. By induction, we have $k-1\geq n-1-s$, and thus $k\geq n-s$, as required.
\end{proof}

The proof of Theorem~\ref{thm:homol_sperner_gen} we propose uses a technique presented in the recent survey by De Loera et al.~\cite[Proposition 2.5]{DeL} for proving Meshulam's lemma.

\begin{proof}[Proof of Theorem~\ref{thm:homol_sperner_gen}]
First, we prove the existence of a map $f_{\sharp}\colon\C(\M,A)\rightarrow\C(\K,A)$ that is augmentation preserving and such that for every $\sigma\in\M$ the support of $f_{\sharp}(\sigma)$ is contained in $\K[\bigcup_{i\in\sigma}V_i]$. We proceed by induction on $k$ and prove that the statement is true for $|\sigma|\leq k$. When $k=0$, we define $f_{\sharp}(i)$ to be any vertex in $V_i$ (which exists because $\H_{-1}(\K[V_i],A)=0$). Suppose now that the statement is true up to $k-1$. For a simplex $\sigma$ such that $|\sigma|=k$, we have $\partial f_{\sharp}(\partial\sigma)=0$ (we apply the induction hypothesis: it is chain map). Since $\H_{k-2}(\K[\bigcup_{i\in\sigma}V_i],A)=0$, there exists an element $f_{\sharp}(\sigma)$ in $C_{k-1}(\K[\bigcup_{i\in\sigma}V_i],A)$ such that $\partial f_{\sharp}(\sigma)=f_{\sharp}(\partial\sigma)$.

Second, define $\lambda\colon V(\K)\rightarrow\{0,\dots,m\}$ by $\lambda(v)=i$ for $v\in V_i$. It induces a simplicial map $\lambda\colon\K\rightarrow\Delta$, where $\Delta$ is the $m$-dimensional simplex with $\{0,\dots,m\}$ as vertex set, and considered as a simplicial complex. Note that $\lambda_{\sharp}$ applied on an $m$-dimensional simplex is nonzero if and only if that simplex is rainbow.  We claim that $(\lambda_{\sharp}\circ f_{\sharp})(\sigma)=\sigma$ for any oriented simplex $\sigma$ of $\C(\M,A)$ (note that $\M$ is a subcomplex of $\Delta$) and we will prove it by induction. Since $f_{\sharp}$ is augmentation-preserving, this is obviously true when $\sigma$ is $0$-dimensional. Take now any oriented simplex $\sigma\in\M$. Since the support of $f_{\sharp}(\sigma)$ is contained in $\K[\bigcup_{i\in\sigma}V_i]$, the chain $(\lambda_{\sharp}\circ f_{\sharp})(\sigma)$ is of the form $x\sigma$ for some $x\in A$. By induction, we have 
$$\partial(\lambda_{\sharp}\circ f_{\sharp})(\sigma)=(\lambda_{\sharp}\circ f_{\sharp})(\partial\sigma)=\partial\sigma.$$ Thus $x\partial\sigma=\partial\sigma$, which means that $x=1$.

Third, consider the chain $z\in C_d(\M,A)$ equal to the sum of all $d$-dimensional oriented simplices of $\M$ (with unitary coefficients) so that $\partial z=0$. Such a chain exists because $\H_d(\M,A)=A$. Now, consider the chain $c'\in C_{d+1}(\K,A)$ defined by $\partial c'=f_{\sharp}(z)$. Such a $c'$ exists because of the condition $\H_d(\K,A)=0$.  We have $\lambda_{\sharp}(\partial c')=z$ since $\lambda_{\sharp}\circ f_{\sharp}$ is the inclusion chain map. According to Lemma~\ref{lem:count} with $n=m+1$ and $s=d+1$,  there are at least $m-d$ simplices in $c=\lambda_{\sharp}(c')$, which means that there exist at least that number of rainbow simplices in $\K$.
\end{proof}

%

\bigskip

\noindent{\bf Acknowledgements.} The authors are grateful to the reviewers for their comments and suggestions that helped improve the paper, especially by simplifying several proofs. They thank Andreas Holmsen for pointing out the paper by Anders Bj\"orner and the similarity of the approach used in the proof of Theorem~\ref{thmunion} with a technique introduced in that paper. Finally, the second author wishes to acknowledge  support  form CONACyT under project 166306,  support from PAPIIT-UNAM under project IN112614.

\bibliographystyle{plain} 
\bibliography{nerve}

\end{document}